\documentclass{amsart}
\usepackage{cases}
\usepackage{amssymb,amsmath}
\usepackage{pifont}
\usepackage{amsfonts}
\usepackage{txfonts}
\usepackage{mathrsfs}
\usepackage{bbm}
\usepackage{enumerate}
\usepackage{color}

\theoremstyle{plain}
\newtheorem{theorem}{Theorem}[section]
\newtheorem{remark}[theorem]{Remark}
\newtheorem{lemma}[theorem]{Lemma}
\newtheorem{proposition}[theorem]{Proposition}

\newtheorem{Example}{Example}[section]
\newtheorem{Definition}{Definition}[section]
\numberwithin{equation}{section}

\begin{document}

\title[Boltzmann equation with Debye-Yukawa potential]
{Cauchy problem to the homogeneous
Boltzmann equation with Debye-Yukawa potential for measure initial datum}

\author[Hao-Guang Li]
{Hao-Guang Li}

\address{Hao-Guang Li,
\newline\indent
School of Mathematics and statistics,
\newline\indent
South-central university for nationalities 430074,Wuhan, P. R. China}
\email{lihaoguang@mail.scuec.edu.cn}

\date{\today}



\begin{abstract}
In this work, we prove the existence, uniqueness and smoothing properties
of the solution to the Cauchy problem for the spatially homogeneous
Boltzmann equation with Debye-Yukawa potential for probability measure initial datum.
\end{abstract}

\maketitle


\section{Introduction}
In this work, we consider the Cauchy problem for the spatially homogeneous Boltzmann equation,
\begin{equation}\label{eq-1}
\left\{ \begin{aligned}
         &\frac{\partial f}{\partial t} = Q(f,f),\,\\
         &f(0,v)=f_0(v).
\end{aligned} \right.
\end{equation}
where $f=f(t,v)$ is the density distribution function depending only on two varia\-bles $t\geq0$~and $v\in\mathbb{R}^{3}$. The Boltzmann bilinear collision operator is given by
\begin{equation*}
Q(g,f)(v)=\int_{\mathbb{R}^{3}}\int_{\mathbb{S}^{2}}B(v-v_{\ast},\sigma)\left(g(v_{\ast}^{\prime})f(v^{\prime})-g(v_{\ast})f(v)\right)dv_{\ast}d\sigma,
\end{equation*}
where for $\sigma\in \mathbb{S}^{2}$,~the symbols~$v_{\ast}^{\prime}$~and~$v^{\prime}$~are abbreviations for the expressions,
$$
v^{\prime}=\frac{v+v_{\ast}}{2}+\frac{|v-v_{\ast}|}{2}\sigma,\,\,\,\,\, v_{\ast}^{\prime}
=\frac{v+v_{\ast}}{2}-\frac{|v-v_{\ast}|}{2}\sigma,
$$
which are obtained in such a way that collision preserves momentum and kinetic energy,~namely
$$
v_{\ast}^{\prime}+v^{\prime}=v+v_{\ast},\,\,\,\,\, |v_{\ast}^{\prime}|^{2}+|v^{\prime}|^{2}=|v|^{2}+|v_{\ast}|^{2}.
$$
For monatomic gas, the collision cross section $B(v-v_{\ast},\sigma)$ is a non-negative function which~depends only on $|v-v_{\ast}|$ and $\cos\theta$
which is defined through the scalar product in $\mathbb{R}^{3}$ by
$$\cos\theta=\frac{v-v_{\ast}}{|v-v_{\ast}|}\cdot\sigma.$$
Without loss of generality, we may assume that $B(v-v_{\ast},\sigma)$ is supported on the set $\cos\theta\geq0,$ i.e. where $0\leq \theta\leq\frac{\pi}{2}$.
See for example \cite{Villani} for more explanations about the support of $\theta$.
For physical models, the collision cross section usually takes the form
\begin{equation*}
B(v-v_{\ast},\sigma)=\Phi(|v-v_{\ast}|)b(\cos\theta).
\end{equation*}
In this paper, we consider only the Maxwellian molecules case with $\Phi\equiv1$.
Except the hard sphere model, the function $b(\cos\theta)$ depends closely on the inter-molecule potentials.
For instance, in the important model case of the inverse-power potentials,
$$
U(\rho)=\frac{1}{\rho^{\gamma-1}}, \,\,\text{with}\,\, \gamma>2,
$$
where $\rho$ denotes the distance between two interacting particles, then
\begin{equation*}
b(\cos\theta)\sin\theta\approx K\theta^{-1-2\nu}, \,\,\text{as}\,\theta\rightarrow0^+, \,\text{where}\, 0<\nu=\frac{1}{\gamma-1}<1.
\end{equation*}
%
If the inter-molecule potential satisfies the Debye-Yukawa type potential,  where the potential function is given by
$$U(\rho)=\frac{1}{\rho\, e^{\rho^s}},\,\,\text{with}\,\,s>0,$$
the collision cross section has a singularity in the following form
\begin{equation}\label{b}
b(\cos\theta)\sim \theta^{-2}\left(\log\Big(\theta/2\Big)^{-1}\right)^{\frac{2}{s}-1},\,\,\text{when}\,\,\theta\rightarrow 0^+,  \,\,\text{with}\,\,s>0.
\end{equation}
This explicit formula was first appeared in the Appendix in \cite{YSCT}.
In some sense, the Debye-Yukawa type potentials is a model between the Coulomb potential corresponding to $s=0$ and the inverse-power potential.
For further details on the physics background and the derivation of the Boltzmann equation, we refer to the references \cite{Cerci}, \cite{Villani}.

In the study of the Cauchy problem of the homogeneous Boltzmann equation in Maxwellian molecules case,
Tanaka in \cite{Tanaka} proved the existence and the uniqueness of the solution under the assumption of the initial data  $f_0>0$,
$$\int_{\mathbb{R}^3}f_0(v)dv=1, \,\int_{\mathbb{R}^3}v_jf_0(v)dv=0,\,j=1,2,3,$$
and
\begin{equation}\label{finite-energy}
\int_{\mathbb{R}^3}|v|^2f_0(v)dv=3.
\end{equation}
The proof of this result was simplified and generalized in \cite{TV} and \cite{Villani1998}.

For the inverse-power potential, Cannone-Karch in \cite{Karch} extended this result for the initial data of the probability measure without \eqref{finite-energy}, this means that the initial data could have infinite energy. Recently, Morimoto \cite{Morimoto2} and Morimoto-Yang \cite{MY} extended this result more profoundly and prove the smoothing effect of the solution to the Cauchy problem \eqref{eq-1} without cutoff assumption in the strong singular case under the measure initial data.
However, the Cauchy problem to the homogeneous Boltzmann equation with Debye-Yukawa potential \eqref{b} has been only studied in \cite{YSCT}. It has been shown in \cite{YSCT} that weak solutions to the Cauchy problem \eqref{eq-1} with Debye-Yukawa
type interactions on $0<s<2$ enjoy $H^{\infty}$ smoothing property, i.e. starting with arbitrary initial datum $f_0\geq0$,
\begin{equation*}
\int_{\mathbb{R}^3}f_0(v)(1+|v|^2+\log(1+f_0(v)))dv<+\infty,
\end{equation*}
one has $f(t,\cdot)\in H^{\infty}(\mathbb{R}^3)$ for any positive time $t > 0$.
The logarithmic regularity theory was first introduced in \cite{Morimoto} on the hypoellipticity of the infinitely degenerate elliptic operator and was developed in \cite{Morimoto-Xu1},\cite{Morimoto-Xu2} on the logarithmic Sobolev estimates.

In the present work, setting the angular function $b$ satisfies the Debye-Yukawa potential \eqref{b} for $s>0$, based upon \cite{Morimoto2} and our recent results of \cite{GL-2015}, \cite{GL-2016} for the Cauchy problem to the linearized homogeneous Boltzmann equation with Debye-Yukawa potential,  we intend to prove the result in \cite{YSCT} for the probability measure initial datum.

Now we introduce the probability measure.
\begin{Definition}
A function $\psi:\mathbb{R}^3\rightarrow\mathbb{C}$ is called a characteristic function if there is a probability measure $\Psi$ (i.e., a positive Borel measure with $\int_{\mathbb{R}^3}d\Psi(v)=1$) such that
the identity $\psi=\int_{\mathbb{R}^3}e^{-iv\cdot\xi}d\Psi(v)$ holds. We denote the set of all characteristic functions by $\mathcal{K}$.
\end{Definition}

Inspired by \cite{TV}, we introduce a subspace $\mathcal{K}^{\alpha}$ for $\alpha\geq 0$ was defined in \cite{Karch} as follows:
\begin{equation*}
\mathcal{K}^{\alpha}=\{\varphi\in\mathcal{K}; \|\varphi-1\|_{\alpha}<+\infty\},
\end{equation*}
where
\begin{equation*}
\|\varphi-1\|_{\alpha}=\sup_{\xi\in\mathbb{R}^3}\frac{|\varphi-1|}{|\xi|^{\alpha}}.
\end{equation*}
The space $\mathcal{K}^{\alpha}$ endowed with the distance
\begin{equation*}
\|\varphi-\psi\|_{\alpha}=\sup_{\xi\in\mathbb{R}^3}\frac{|\varphi-\psi|}{|\xi|^{\alpha}}
\end{equation*}
is a complete metric space (see Proposition 3.10 in \cite{Karch}).

It follows that
$\mathcal{K}^{\alpha}={1}$ for $\alpha>2$ and the embeddings (Lemma 3.12 of \cite{Karch})
$${1}\subset\mathcal{K}^{\alpha}\subset\mathcal{K}^{\beta}\subset\mathcal{K}^0=\mathcal{K},\,\text{for all} \,2\geq\alpha\geq\beta\geq0.$$

In this paper, we consider the Cauchy problem \eqref{eq-1} for the initial datum of the probability measure $\Psi_0(v)$.
If we set $\psi_0(\xi)=\int_{\mathbb{R}^3}e^{-iv\cdot\xi}d\Psi_0(v)$ and denote the Fourier transform of the probability measure solution by $\psi(t,\xi)$, then it follows from the Bobylev formula in \cite{Boby} that the Cauchy problem \eqref{eq-1} is reduced to
\begin{equation}\label{eq-2}
\left\{ \begin{aligned}
         &\partial_t\psi(t,\xi) = \int_{\mathbb{S}^2}b\left(\frac{\xi\cdot\sigma}{|\xi|}\right)\left(\psi(t,\xi^+)\psi(t,\xi^-)-\psi(t,\xi)\psi(t,0)\right)d\sigma,\,\\
         &\psi(0,\xi)=\psi_0(\xi),
\end{aligned} \right.
\end{equation}
where $\xi^\pm=\frac{\xi}{2}\pm\frac{|\xi|}{2}\sigma.$

\begin{theorem}\label{trick}
The Maxwellian collision cross-section $b(\,\cdot\,)$ satisfies the assumption $\eqref{b}$ with $s>0$.
Then for any $\alpha>0$ and every $\psi_0\in\mathcal{K}^{\alpha}$,
there exists a unique classical solution $\psi\in C([0,+\infty),\mathcal{K}^{\alpha})$ of the Cauchy problem \eqref{eq-2}.
Furthermore, let $\psi(t,\xi),\, \varphi(t,\xi)\in C([0,+\infty),\mathcal{K}^{\alpha})$ be two solutions to
the Cauchy problem \eqref{eq-2} with the initial datum $\psi_0, \varphi_0\in\mathcal{K}^{\alpha}$, then for any $t>0$, we have
\begin{equation}\label{stable}
\|\psi(t)-\varphi(t)\|_{\alpha}\leq e^{\lambda_{\alpha}t}\|\psi_0-\varphi_0\|_{\alpha},
\end{equation}
where
\begin{equation}\label{lambda}
\lambda_{\alpha}=2\pi\int^{\frac{\pi}{2}}_0\beta(\theta)(\cos^\alpha\frac{\theta}{2}+\sin^\alpha\frac{\theta}{2}-1)d\theta
\end{equation}
\end{theorem}

\begin{remark}
Comparing with the restriction stated in \cite{Karch} and \cite{Morimoto2} that: for some $\alpha_0>0$, $\alpha\in [\alpha_0,2]$ satisfies
$$\sin^{\alpha_0}\frac{\theta}{2} b(\cos\theta)\sin\theta\in L^1([0,\frac{\pi}{2}]),$$
we study the Cauchy problem \eqref{eq-2} without this condition and instead, we set the assumption for any $\alpha>0$.  Because for the collision kernel given in \eqref{b}, $b(\,\cdot\,)$ satisfies
\begin{align}\label{kernel}
&\int^{\frac{\pi}{2}}_0\sin^{\alpha}\frac{\theta}{2} b(\cos\theta)\sin\theta d\theta\nonumber\\
\lesssim&\int^{1}_0u^{\alpha-1}(\log u^{-1})^{\frac{2}{s}-1}du=\int^{+\infty}_{1}x^{-\alpha-1}(\log x)^{\frac{2}{s}-1}dx\nonumber\\
=&\int^{+\infty}_0e^{-\alpha u}u^{\frac{2}{s}-1}du=\alpha^{-\frac{2}{s}}\Gamma(\frac{2}{s})
\end{align}
where $\Gamma(\frac{2}{s})=\int^{+\infty}_0x^{\frac{2}{s}-1}e^{-x}dx$ is the standard Gamma function.
\end{remark}
The regularity of the  Boltzmann equation has been studied in many works. Under the assumption of the singularity of the collision kernel $b(\,\cdot\,)$, we have the smoothing effect of solutions to the Cauchy problem for the spatially homogenous Boltzmann equation for the initial data $f_0>0$ satisfies
$$\int_{\mathbb{R}^3}f_0(v)(1+|v|^2)dv<+\infty,\,\,\int_{\mathbb{R}^3}f_0(1+\log f_0)dv<+\infty,$$
see \cite{Villani}, \cite{ADVW}, \cite{YSCT} and the references therein.
However, we can not always expect the smoothing effect for solutions to the Cauchy problem for the spatially homogeneous Boltzmann equation in the probability measures whose Fourier transforms are in $\mathcal{K}^{\alpha}$.  Since $1\in\mathcal{K}^{\alpha}$, is the Fourier transform of the Dirac mass on 0.
Besides, we present an example of the smoothing property for the Cauchy problem of the spatially homogeneous Boltzmann equation with measure initial datum:
\begin{Example}\label{example}
Put $\mathbf{e}_1=(1,0,0),\,\mathbf{e}_2=(0,1,0),\,\mathbf{e}_3=(0,0,1)$, then the set of vectors $\{\mathbf{e}_k\}_{k=1,2,3}$ forms an orthonormal basis of  $\mathbb{R}^3$.
Let
$$f_0=\frac{1}{2}\frac{1}{\sqrt{2\pi}}e^{-\frac{|v|^2}{2}}+\frac{1}{12}\sum^3_{k=1}(\delta(v-\mathbf{e}_k)+\delta(v+\mathbf{e}_k)),$$
if $b(\,\cdot\,)$ satisfies \eqref{b} with $0<s<2$ and $f(t,v)$ is the unique solution with initial data $f_0$, then $f(t,v)\in H^{\infty}(\mathbb{R}^3)$ for  $t\in (0,T]$ with some $T>0$.
\end{Example}

To interpret this example, we need to prove the following result of  $H^{\infty}$ smoothing effect given in \cite{YSCT} (see also \cite{AE}).
\begin{proposition}\label{trick1}
Assume that $b(\,\cdot\,)$ is given in~$\eqref{b}$ with $0<s<2$.
$\psi\in C([0,+\infty),\mathcal{K}^{\alpha})$ for any $\alpha>0$ is a unique solution of the Cauchy problem \eqref{eq-2}, if for any $T>0$, there exists a constant $D_T>0$ such that the solution $\psi(t,\xi)$ satisfies
\begin{equation}\label{regular-kernel}
\inf_{t\in[0,T]}\left(1-|\psi(t,\xi)|\right)\geq D_T\min(1,|\xi|^2),
\end{equation}
then the inverse Fourier transform of $\psi(t,\xi)$ belongs to $H^{\infty}(\mathbb{R}^3)$ for any $t\in(0,T].$
\end{proposition}

\begin{remark}
The inequality is a key for the coercive estimate for the smoothing effect of the Cauchy problem for the non-cutoff homogeneous Boltzmann equation, see (2.2) of \cite{YSCT}, we also refer the readers to \cite{HMUY} and the references therein. In fact, for the initial data $f(t)>0$ satisfies
$$\int_{\mathbb{R}^3}f(t)(1+|v|^2)dv<+\infty,\,\,\int_{\mathbb{R}^3}f(t)(1+\log f(t))dv<+\infty,$$
then the Fourier transform of $f$ satisfies \eqref{regular-kernel} (see Lemma 3 of \cite{ADVW}).
\end{remark}

The rest of the paper is arranged as follows. The proof of Theorem \ref{trick} will be presented in Section \ref{S2}.
In Section \ref{S3}, we will prove Proposition \ref{trick1} and  show the $H^{\infty}$ smoothing effect for the Example \ref{example}.

\section{The proof of Theorem \ref{trick} }\label{S2}
The construction of the solution to the Cauchy problem for the homogeneous Boltzmann equation with cutoff assumption has been done in Section 4 of \cite{Karch}.  Our idea is:  Constructing a sequence solutions under the cutoff assumption, limiting the sequence solutions in a suitable space, then proving the limit solution is the solution under the non-cutoff assumption.  So the difficult part of the proof is to show the uniqueness part of the theorem \ref{trick}.

The following Lemmas are used for the proof of the uniqueness part of the theorem \ref{trick}.
\begin{lemma}
For any characteristic function $\varphi\in\mathcal{K}$, we have
\begin{equation}\label{char}
|\varphi(\xi)\varphi(\eta)-\varphi(\xi+\eta)|^2\leq(1-|\varphi(\xi)|^2)(1-|\varphi(\eta)|^2).
\end{equation}
for all $\xi,\eta\in \mathbb{R}^3$ and moreover if $\varphi\in\mathcal{K}^{\alpha})$, then
\begin{equation}\label{charalpha}
|\varphi(\xi)-\varphi(\xi+\eta)|\leq\|\varphi-1\|_{\alpha}(4|\xi|^{\frac{\alpha}{2}}|\eta|^{\frac{\alpha}{2}}+|\eta|^{\alpha}).
\end{equation}
\end{lemma}
\begin{proof}
The proof of \eqref{char}, we can refer to $(18)$ in Lemma 2.1 of \cite{Morimoto2}, $(3.5)$ of \cite{Karch} and also Lemma 3.5.10 of \cite{Jacob}.
The proof of \eqref{charalpha} refer to $(19)$ in Lemma 2.1 of \cite{Morimoto2}.
\end{proof}

By a proof similar to the Lemma 2.2 in \cite{Morimoto2}, we obtain the following Lemma.
\begin{lemma}\label{ke}
Assume that $b(\,\cdot\,)$ is given in~$\eqref{b}$ with $s>0$, for $\varphi\in\mathcal{K}^{\alpha}$ and for any $\alpha>0$, we have
\begin{equation}\label{re}
\left|\int_{\mathbb{S}^2}b(\frac{\xi\cdot\sigma}{|\xi|})\left(\varphi(\xi^+)\varphi(\xi^-)-\varphi(\xi)\right)d\sigma\right|
\lesssim\alpha^{-\frac{2}{s}}\Gamma(\frac{2}{s})\|1-\varphi\|_{\alpha}|\xi|^{\alpha}.
\end{equation}
\end{lemma}
\begin{proof}
Put $\zeta=(\xi^+\cdot\frac{\xi}{|\xi|})\frac{\xi}{|\xi|}$, then we set $\tilde{\xi}^+=2\zeta-\xi^+$, which is symmetric to $\xi^+$ with respect to $\xi$.
We can divide the integral on the left hand side into three parts,
\begin{align*}
&\int_{\mathbb{S}^2}b(\frac{\xi\cdot\sigma}{|\xi|})\left(\varphi(\xi^+)\varphi(\xi^-)-\varphi(\xi)\right)d\sigma\\
=&\frac{1}{2}\int_{\mathbb{S}^2}b(\frac{\xi\cdot\sigma}{|\xi|})\left(\varphi(\xi^+)+\varphi(\tilde{\xi}^+)
-2\varphi(\zeta)\right)d\sigma+\int_{\mathbb{S}^2}b(\frac{\xi\cdot\sigma}{|\xi|})\left(\varphi(\zeta)-\varphi(\xi)\right)d\sigma\\
&+\int_{\mathbb{S}^2}b(\frac{\xi\cdot\sigma}{|\xi|})\varphi(\xi^+)\left(\varphi(\xi^-)-1\right)d\sigma\\
=&I_1+I_2+I_3.
\end{align*}
Since in part $I_1$,
\begin{align*}
|\varphi(\xi^+)+\varphi(\tilde{\xi}^+)-2\varphi(\zeta)|
&=|\int_{\mathbb{R}^3}e^{-i\zeta\cdot v}(e^{-\eta^+\cdot v}+e^{-\eta^-\cdot v}-2)d\Psi(v)|\\
&\leq\int_{\mathbb{R}^3}(2-e^{-\eta^+\cdot v}+e^{-\eta^-\cdot v})d\Psi(v)\\
&\leq2\|1-\varphi\|_{\alpha}|\xi|^{\alpha}(\sin\theta/2)^{\alpha}.
\end{align*}
For $I_2$, by using the formula \eqref{charalpha}, we have
\begin{align*}
|\varphi(\zeta)-\varphi(\xi)|
&\leq\|1-\varphi\|_{\alpha}(4|\xi|^{\frac{\alpha}{2}}|\xi-\zeta|^{\frac{\alpha}{2}}+|\xi-\zeta|^{\alpha})\\
&\leq5\|1-\varphi\|_{\alpha}|\xi|^{\alpha}(\sin\theta/2)^{\alpha},
\end{align*}
and for $I_3$, using the elementary equality $|\varphi(\xi^+)|\leq\varphi(0)=1$ in Lemma 3.11 in \cite{Karch},
\begin{align*}
&|\varphi(\xi^+)(\varphi(\xi^-)-1)|\leq\|1-\varphi\|_{\alpha}|\xi|^{\alpha}(\sin\theta/2)^{\alpha}.
\end{align*}
Therefore, we have
\begin{align*}
&\left|\int_{\mathbb{S}^2}b(\frac{\xi\cdot\sigma}{|\xi|})\left(\varphi(\xi^+)\varphi(\xi^-)-\varphi(\xi)\right)d\sigma\right|\\
\leq& 16\pi\left(\int^{\frac{\pi}{2}}_0\sin^{\alpha}\frac{\theta}{2} b(\cos\theta)\sin\theta d\theta\right)\|1-\varphi\|_{\alpha}|\xi|^{\alpha}.
\end{align*}
The formula \eqref{re} follows from the above inequality and \eqref{kernel}.
\end{proof}

In fact, from the proof of Lemma \ref{ke}, we can lead to an intuitive understanding.
\begin{lemma}\label{direct}
Let $b(\,\cdot\,)$ be the function given in~$\eqref{b}$ with $s>0$. for $\varphi\in\mathcal{K}^{\alpha}$, $\alpha>0$ and for any $\epsilon>0$, set
$$\Omega_{\epsilon}=\Omega_{\epsilon}(\xi)=\left\{\sigma\in\mathbb{S}^2;1-\frac{\xi}{|\xi|}\cdot\sigma\leq 2\left(\frac{\epsilon}{\pi}\right)^2\right\}$$
and
$$R_{\epsilon,\varphi}(\xi)=\int_{\mathbb{S}^2\cap\Omega_{\epsilon}}b(\frac{\xi\cdot\sigma}{|\xi|})
\frac{\left(\varphi(\xi^+)\varphi(\xi^-)-\varphi(\xi)\right)}{|\xi|^{\alpha}}d\sigma,$$
then we obtain,
\begin{align*}
|R_{\epsilon,\varphi}(\xi)|\lesssim \alpha^{-\frac{2}{s}}\|1-\varphi\|_{\alpha}\left(\int^{+\infty}_{\log\big(\frac{1}{\epsilon}\big)}u^{\frac{2}{s}-1}e^{-u}du\right)\rightarrow0\,\,\text{as}\,\,\epsilon\rightarrow0^+.
\end{align*}
\end{lemma}

Now we are prepared to prove Theorem \ref{trick}.  The proof of this Theorem is mostly the same as in \cite{Morimoto2}.
\begin{proof}[\bf{The proof of the Uniqueness}]
For $\alpha>0$, let $\psi(t,\xi),\, \varphi(t,\xi)\in C([0,+\infty),\mathcal{K}^{\alpha})$ be two solutions to
the Cauchy problem \eqref{eq-2} with the initial datum $\psi_0, \varphi_0\in\mathcal{K}^{\alpha}$.
Set
$$h(t,\xi)=\frac{\psi(t,\xi)-\varphi(t,\xi)}{|\xi|^{\alpha}},$$
it follows that,
\begin{align*}
\partial_th(t,\xi)
=&\int_{\mathbb{S}^2\cap\Omega_{\epsilon}^c}
b(\frac{\xi}{|\xi|}\cdot\sigma)
\frac{\psi(t,\xi^+)\psi(t,\xi^-)-\varphi(t,\xi^+)\varphi(t,\xi^-)}{|\xi|^{\alpha}}d\sigma\\
&-\left(\int_{\mathbb{S}^2\cap\Omega_{\epsilon}^c}
b(\frac{\xi}{|\xi|}\cdot\sigma)d\sigma\right) h(t,\xi)\\
&+\int_{\mathbb{S}^2\cap\Omega_{\epsilon}}
b(\frac{\xi}{|\xi|}\cdot\sigma)\frac{\psi(t,\xi^+)\psi(t,\xi^-)-\psi(t,\xi)}{|\xi|^{\alpha}}d\sigma\\
&-\int_{\mathbb{S}^2\cap\Omega_{\epsilon}}
b(\frac{\xi}{|\xi|}\cdot\sigma)
\frac{\varphi(t,\xi^+)\varphi(t,\xi^-)-\varphi(t,\xi)}{|\xi|^{\alpha}}d\sigma\\
=&I_{\epsilon}(t,\xi)-a_{\epsilon}h(t,\xi)
+R_{\epsilon,\psi}(t,\xi)-R_{\epsilon,\varphi}(t,\xi)
\end{align*}
where
\begin{align*}
a_{\epsilon}&=\int_{\mathbb{S}^2\cap\Omega_{\epsilon}^c}
b(\frac{\xi}{|\xi|}\cdot\sigma)d\sigma=2\pi\int^{\frac{\pi}{2}}_{2\arcsin\frac{\epsilon}{\pi}}
b(\theta)\sin\theta d\theta\\
&\sim\int^{\frac{\pi}{2}}_{2\arcsin\frac{\epsilon}{\pi}}\theta^{-1}(\log \theta^{-1})^{\frac{2}{s}-1}d\theta,
\end{align*}
it diverges as $\epsilon\rightarrow0^+.$  Let $R>0$, for any $\xi$ with $|\xi|\leq R$,
\begin{align*}
&\left|\frac{\psi(t,\xi^+)\psi(t,\xi^-)
-\varphi(t,\xi^+)\varphi(t,\xi^-)}{|\xi|^{\alpha}}\right|\\
=&\left|\frac{\psi(t,\xi^+)(\psi(t,\xi^-)-\varphi(t,\xi^-))}{|\xi|^{\alpha}}
+\frac{\varphi(t,\xi^-)(\psi(t,\xi^+)-\varphi(t,\xi^+))}{|\xi|^{\alpha}}\right|\\
=&\left|\psi(t,\xi^+)h(t,\xi^-)\frac{|\xi^-|^{\alpha}}{|\xi|^{\alpha}}
+\varphi(t,\xi^-)h(t,\xi^+)\frac{|\xi^+|^{\alpha}}{|\xi|^{\alpha}}\right|\\
\end{align*}
Applying the elementary inequality
$|\psi(t,\xi^+)|\leq\psi(t,0)=1, |\varphi(t,\xi^-)|\leq\varphi(t,0)=1$ in Lemma 3.11 of \cite{Karch} again and the fact that for $0<\theta<\frac{\pi}{2}$, $|\xi^+|=|\xi|\cos\frac{\theta}{2}, |\xi^-|=|\xi|\sin\frac{\theta}{2}\leq |\xi|$,  we obtain,
\begin{align*}
\left|\frac{\psi(t,\xi^+)\psi(t,\xi^-)
-\varphi(t,\xi^+)\varphi(t,\xi^-)}{|\xi|^{\alpha}}\right|\leq
H_R(t)
\left(\cos^{\alpha}\frac{\theta}{2}+\sin^{\alpha}\frac{\theta}{2}\right),
\end{align*}
where $H_R(t)=\sup_{|\xi|\leq R}|h(t,\xi)|.$  In fact, we have
\begin{align*}
|I_{\epsilon}(t,\xi)|\leq \lambda_{\epsilon,\alpha}H_R(t)
\end{align*}
where
$$\lambda_{\epsilon,\alpha}=2\pi\int^{\frac{\pi}{2}}_{2\arcsin\frac{\epsilon}{\pi}}
b(\theta)\sin\theta\left(\cos^{\alpha}\frac{\theta}{2}+\sin^{\alpha}\frac{\theta}{2}\right)d\theta .$$
Notice that, as $\epsilon\rightarrow0,$
\begin{align*}
\lambda_{\epsilon,\alpha}-a_{\epsilon}\rightarrow 2\pi\lambda_{\alpha}=\int^{\frac{\pi}{2}}_{0}
b(\theta)\sin\theta\left(\cos^{\alpha}\frac{\theta}{2}+\sin^{\alpha}\frac{\theta}{2}-1\right)d\theta
\end{align*}
where $\lambda_{\alpha}$ was given in \eqref{lambda}.
Since $\psi(t,\xi),\, \varphi(t,\xi)\in C([0,+\infty),\mathcal{K}^{\alpha}),$ it follows from Lemma \ref{direct} that for any fixed $T>0$,
$$
\sup_{t\in(0,T]}
\left(|R_{\epsilon,\psi}(t,\xi)|
+|R_{\epsilon,\varphi}(t,\xi)|\right)=r_{\epsilon}\rightarrow0,\,\text{as}
\,\epsilon\rightarrow0^+.
$$
Therefore, we obtain that, for $|\xi|\leq R,$
$$|\partial_th(t,\xi)+a_{\epsilon}h(t,\xi)|\leq \lambda_{\epsilon,\alpha}H_R(t)+r_{\epsilon}.$$
Integrating from $0$ to $t$,
\begin{align*}
\left|e^{a_{\epsilon}t}h(t,\xi)-h(0,\xi)\right|
&=\left|\int^t_0\frac{\partial}{\partial_{\tau}}(e^{a_{\epsilon}\tau}h(\tau,\xi))d\tau\right|\\
&\leq\int^t_0\left|\frac{\partial}{\partial_{\tau}}\Big(e^{a_{\epsilon}\tau}h(\tau,\xi)\Big)\right|d\tau
=\int^t_0 e^{a_{\epsilon}\tau}\left|\partial_{\tau}h(\tau,\xi)+a_{\epsilon}h(\tau,\xi)\right|d\tau\\
&\leq\int^t_0 e^{a_{\epsilon}\tau}
(\lambda_{\epsilon,\alpha}H_R(\tau)+r_{\epsilon})
d\tau.
\end{align*}
Then it follows that,
\begin{align}\label{Gron}
e^{a_{\epsilon}t}H_R(t)&\leq H_R(0)
+\int^t_0 e^{a_{\epsilon}\tau}
(\lambda_{\epsilon,\alpha}H_R(\tau)+r_{\epsilon})
d\tau\nonumber\\
&=\lambda_{\epsilon,\alpha}\int^t_0 e^{a_{\epsilon}\tau}H_R(\tau)d\tau
+\frac{e^{a_{\epsilon}t}-1}{a_{\epsilon}}r_{\epsilon}+H_R(0).
\end{align}
Let $\eta(t)=\int^t_0 e^{a_{\epsilon}\tau}H_R(\tau)d\tau$, then
$$\eta'(t)\leq\lambda_{\epsilon,\alpha}\eta(t)+
\frac{e^{a_{\epsilon}t}-1}{a_{\epsilon}}r_{\epsilon}+H_R(0).$$
According to the Gronwall's inequality:
\begin{align*}
\eta(t)
&\leq e^{\lambda_{\epsilon,\alpha}t}
\int^t_0e^{-\lambda_{\epsilon,\alpha}\tau}\left[\frac{e^{a_{\epsilon}\tau}-1}{a_{\epsilon}}r_{\epsilon}+H_R(0)\right]d\tau\\
&=\frac{e^{a_{\epsilon}t}-e^{\lambda_{\epsilon,\alpha}t}}{(a_{\epsilon}-\lambda_{\epsilon,\alpha})a_{\epsilon}}r_{\epsilon}
+\frac{e^{\lambda_{\epsilon,\alpha}t}-1}{\lambda_{\epsilon,\alpha}a_{\epsilon}}r_{\epsilon}+\frac{e^{\lambda_{\epsilon,\alpha}t}-1}{\lambda_{\epsilon,\alpha}}H_R(0).
\end{align*}
Substituting into the inequality \eqref{Gron},  we have,
\begin{align*}
e^{a_{\epsilon}t}H_R(t)
&\leq\frac{e^{a_{\epsilon}t}-e^{\lambda_{\epsilon,\alpha}t}}
{a_{\epsilon}-\lambda_{\epsilon,\alpha}}r_{\epsilon}
+\frac{2e^{\lambda_{\epsilon,\alpha}t}-2}{a_{\epsilon}}r_{\epsilon}
+e^{\lambda_{\epsilon,\alpha}t}H_R(0).
\end{align*}
We conclude that,
\begin{align*}
H_R(t)
&\leq e^{(\lambda_{\epsilon,\alpha}-a_{\epsilon})t}H_R(0)
+\left[\frac{(e^{(\lambda_{\epsilon,\alpha}-a_{\epsilon})t}-1)}{\lambda_{\epsilon,\alpha}-a_{\epsilon}}
+\frac{2e^{-a_{\epsilon}t}(e^{\lambda_{\epsilon,\alpha}t}-1)}{a_{\epsilon}}\right]r_{\epsilon}\\
&\leq\,e^{(\lambda_{\epsilon,\alpha}-a_{\epsilon})t}H_R(0)
+\left[\frac{(e^{(\lambda_{\epsilon,\alpha}-a_{\epsilon})t}-1)}{\lambda_{\epsilon,\alpha}-a_{\epsilon}}
+\frac{2e^{(\lambda_{\epsilon,\alpha}-a_{\epsilon})t}}{a_{\epsilon}}\right]r_{\epsilon}.
\end{align*}
Because $\lambda_{\epsilon,\alpha}-a_{\epsilon}\rightarrow\lambda_{\alpha},\,a_{\epsilon}\rightarrow+\infty$
and $r_{\epsilon}\rightarrow0$ as
$\epsilon\rightarrow0^+$, taking the limit $\epsilon\rightarrow0^+$, we can deduce that
$$H_R(t)\leq e^{\lambda_{\alpha}t}H_R(0).$$
Taking the limit $R\rightarrow+\infty$, we have
$$\|\psi(t)-\varphi(t)\|_{\alpha}=\sup_{\xi\in\mathbb{R}^3}\frac{|\psi(t,\xi)-\varphi(t,\xi)|}{|\xi|^{\alpha}}
\leq e^{\lambda_{\alpha}t}\|\psi_0-\varphi_0\|_{\alpha},$$
Therefore, we obtain \eqref{stable}.  This ends the proof of the uniqueness.
\end{proof}

\begin{proof}[\bf{The proof of the existence}]
Firstly, constructing a sequence solutions under the cutoff assumption
$$b_n(\cos\theta)=\min(b(\cos\theta),n)\leq b(\cos\theta), \text{for}\,n\in\mathbb{N},$$
then by the result in Section 4 of \cite{Karch}, there exists a solution $\psi_n(t,\xi)\in C([0,+\infty),\mathcal{K}^{\alpha})$ to the Cauchy problem \eqref{eq-2}. By a proof similar to that in \cite{Morimoto2}, we have
$\{\psi_n(t,\xi)\}_{n\in\mathbb{N}}$ is equi-continuity and uniformly bounded in $[0,T]\times\mathbb{R}^3$.
By the Ascoli-Arzel\`a theorem, there exists a convergent sequence  $\{\psi_{n_k}(t,\xi)\}_{k\in\mathbb{N}}$, such that
$\lim_{k\rightarrow +\infty}\psi_{n_k}(t,\xi)=\psi(t,\xi)$ is the solution of \eqref{eq-2}.
This ends the proof of the existence.
\end{proof}

\section{The proof of the Proposition \ref{trick1} }\label{S3}
In this section, we prove the regularity of the solution to Cauchy problem \eqref{eq-2}.  Notice that, to prove the the regularity, we only assume $0<s<2$.  For $s\geq2$, we can't get any smoothing effect of the solution to the Cauchy problem \eqref{eq-2}.

Before the proof of the regularity of the solution to Cauchy problem \eqref{eq-2}, we present the coercive estimate for the kernel.  The proof is similar in spirit to Lemma 4 of \cite{ADVW}.
\begin{lemma}\label{regular}
The collision kernel $b(\,\cdot\,)$ satisfies the assumption $\eqref{b}$ with $0<s<2$, namely,
$$b(\cos\theta)\sin\theta\sim \theta^{-1}\left(\log\Big(\frac{\theta}{2}\Big)^{-1}\right)^{\frac{2}{s}-1}$$
then for $\xi\in\mathbb{R}^3$,
$$\int_{\mathbb{S}^2}
b(\frac{\xi}{|\xi|}\cdot\sigma)\min(1,|\xi^-|^2)d\sigma\gtrsim(\log\langle|\xi|\rangle)^{\frac{2}{s}}.$$
\end{lemma}
\begin{proof}
Since $b(\,\cdot\,)$ satisfies the assumption $\eqref{b}$, for $|\xi|>2$
\begin{align*}
&\int_{\mathbb{S}^2}
b(\frac{\xi}{|\xi|}\cdot\sigma)\min(1,|\xi^-|^2)d\sigma\\
&\gtrsim\int^{\frac{\pi}{4}}_0\theta^{-1}(\log\theta^{-1})^{\frac{2}{s}-1}\min(1,|\xi|^2\theta^2)d\theta\\
&\gtrsim\int^{\frac{\pi}{4}}_{\frac{1}{|\xi|}}\theta^{-1}(\log\theta^{-1})^{\frac{2}{s}-1}d\theta
=\int^{|\xi|}_{\frac{4}{\pi}}(\log u)^{\frac{2}{s}-1}\frac{du}{u}\\
&=\int^{\log(|\xi|)}_{\log(\frac{4}{\pi})} u^{\frac{2}{s}-1}du=\frac{s}{2}[(\log|\xi|)^{\frac{2}{s}}-(\log(\frac{4}{\pi}))^{\frac{2}{s}}]\gtrsim (\log\langle|\xi|\rangle)^{\frac{2}{s}}.
\end{align*}
On the other hand, for $|\xi|\leq2$,
\begin{align*}
&\int^{\frac{\pi}{4}}_0\theta^{-1}(\log\theta^{-1})^{\frac{2}{s}-1}\min(1,|\xi|^2\theta^2)d\theta\\
&\gtrsim \left(\int^{\frac{\pi}{4}}_0\theta(\log\theta^{-1})^{\frac{2}{s}-1}d\theta\right)|\xi|^2\gtrsim(\log\langle|\xi|\rangle)^{\frac{2}{s}}.
\end{align*}
we conclude that, for $\xi\in\mathbb{R}^3$,
\begin{equation*}
\int_{\mathbb{S}^2}
b(\frac{\xi}{|\xi|}\cdot\sigma)\min(1,|\xi^-|^2)d\sigma\gtrsim(\log\langle|\xi|\rangle)^{\frac{2}{s}}.
\end{equation*}
This ends the proof of Lemma \ref{regular}.
\end{proof}

Now we are prepared to proof the Proposition \ref{trick1}.
\begin{proof}[\bf{The proof of the Proposition \ref{trick1}}]
As in  \cite{YSCT}, \cite{Morimoto2} and the references therein, set the time dependent weight function
$$M_{\delta}(t,\xi)=\langle\xi\rangle^{Nt-4}
\langle\delta\xi\rangle^{-2N_0}, \,\text{with}\,\,\langle\xi\rangle^2=1+|\xi|^2$$
where $N_0=\frac{NT}{2}+2$, $N\in\mathbb{N}$ and $\delta>0$ is a small positive constant. Multiplying $M_{\delta}(t,\xi)^2\overline{\psi(t,\psi)}$ by the equation \eqref{eq-2} and integrate over $\mathbb{R}^3$,
We define
$$\psi^{\pm}=\psi(t,\xi^{\pm});\,\,M^+=M_{\delta}(t,\xi^+),$$
then
\begin{equation}\label{equality}
2\int_{\mathbb{R}^3}Re\left(\partial_t\psi M^2\overline{\psi}\right)d\xi
-2\int_{\mathbb{R}^3\times \mathbb{S}^2}b(\frac{\xi}{|\xi|}\cdot\sigma)Re\{(\psi^+\psi^--\psi)M^2\overline{\psi}\}d\sigma d\xi=0.
\end{equation}
Consider the second term
\begin{align*}
-2Re\{(\psi^+\psi^--\psi)M^2\overline{\psi}\}
=&\left(|M\psi|^2+|M^+\psi^+|^2-
2Re\{\psi^-M^+\psi^+\overline{M\psi}\}\right)\\
&+\left(|M\psi|^2-|M^+\psi^+|^2\right)
+2Re\{\psi^-(M-M^+)\psi^+\overline{M\psi}\}\\
=&J_1+J_2+J_3
\end{align*}
For the term $J_1$, by using the Cauchy-Schwarz inequality $$\left|2Re\{\psi^-M^+\psi^+\overline{M\psi}\}\right|\geq -|\psi^-|(|M\psi|^2+|M^+\psi^+|^2),$$
we obtain from the definition of \eqref{regular-kernel} that
$$J_1\geq D_T\min(1,|\xi^-|^2)|M\psi|^2.$$
Then for $b(\,\cdot\,)$ satisfies the assumption $\eqref{b}$,
\begin{align*}
&\int_{\mathbb{R}^3\times\mathbb{S}^2}
b(\frac{\xi}{|\xi|}\cdot\sigma)J_1d\sigma d\xi\\
\gtrsim&\int_{\mathbb{R}^3}\left(\int_{\mathbb{S}^2}
b(\frac{\xi}{|\xi|}\cdot\sigma)\min(1,|\xi^-|^2)d\sigma \right) |M\psi|^2d\xi\\
\gtrsim&\int_{\mathbb{R}^3}\left(\int^{\frac{\pi}{4}}_0\theta^{-1}(\log\theta^{-1})^{\frac{2}{s}-1}\min(1,|\xi|^2\theta^2)d\theta\right) |M\psi|^2d\xi.
\end{align*}
We deduce from Lemma \ref{regular} that,
\begin{equation*}
\int_{\mathbb{R}^3\times\mathbb{S}^2}
b(\frac{\xi}{|\xi|}\cdot\sigma)J_1d\sigma d\xi\gtrsim \int_{\mathbb{R}^3}(\log\langle|\xi|\rangle)^{\frac{2}{s}}|M\psi|^2d\xi.
\end{equation*}
Using the change of variable $\xi\rightarrow\xi^+$ for the term $M^+\psi^+$ in $J_2$, in spirt of the cancellation lemma $($see Lemma 1 of \cite{ADVW}$)$, we have
\begin{align*}
&\left|\int_{\mathbb{R}^3\times\mathbb{S}^2}
b(\frac{\xi}{|\xi|}\cdot\sigma)J_2d\sigma d\xi\right|\\
&\lesssim\int^{\frac{\pi}{4}}_0\theta(\log\theta^{-1})^{\frac{2}{s}-1}d\theta\int_{\mathbb{R}^3}|M\psi|^2d\xi\\
&\lesssim \int_{\mathbb{R}^3}|M\psi|^2d\xi.
\end{align*}
For the last term $J_3$, we observe that $|M-M^+|\lesssim\sin^2\theta M^+$, cf. $(3.4)$ in \cite{YSCT},
then
\begin{align*}
|\int_{\mathbb{R}^3\times\mathbb{S}^2}
b(\frac{\xi}{|\xi|}\cdot\sigma)J_3d\sigma d\xi|\lesssim \int_{\mathbb{R}^3}|M\psi|^2d\xi.
\end{align*}
Finally, substituting these estimations of $J_1, J_2, J_3$ back to \eqref{equality}, we have for a constant $c_0>0,$
such that,
\begin{align*}
\frac{d}{dt}\int_{\mathbb{R}^3}|M\psi(t,\xi)|^2d\xi
+\int_{\mathbb{R}^3}\left(c_0(\log\langle|\xi|\rangle)^{\frac{2}{s}}-2N\log\langle|\xi|\rangle\right)|M\psi|^2d\xi
\lesssim \int_{\mathbb{R}^3}|M\psi|^2d\xi.
\end{align*}
Since for $0<s<2$,
$$(\log\langle|\xi|\rangle)^{\frac{2}{s}-1}\rightarrow +\infty\,\,\text{as}\,\,|\xi|\rightarrow+\infty,$$
we have
\begin{align*}
\frac{d}{dt}\int_{\mathbb{R}^3}|M\psi(t,\xi)|^2d\xi
\lesssim \int_{\mathbb{R}^3}|M\psi|^2d\xi.
\end{align*}
It follows from the Gronwall inequality that, for any $t\in[0,T]$,
$$\int_{\mathbb{R}^3}|\langle\xi\rangle^{Nt-4}(1+\delta^2|\xi|^2)^{-N_0}\psi(t,\xi)|^2d\xi
\lesssim\int_{\mathbb{R}^3}|\langle\xi\rangle^{-4}\psi_0|^2d\xi\lesssim\|\psi_0\|_{\alpha}.$$
Let $\delta\rightarrow0$ and N be an arbitrarily large, we ends the proof of the regularity.
\end{proof}

We prove the $H^{+\infty}$ smoothing effect of the solution in Example \ref{example},
remark that $f_0$ in our Example satisfying some basic equality of
$$\int_{\mathbb{R}^3}f_0(v)dv=1, \,\int_{\mathbb{R}^3}v_jf_0(v)dv=0,\,j=1,2,3.$$

\begin{proof}[\bf{The proof of the Example \ref{example} }]
Let $\psi_0$ and $\psi(t)$ be the Fourier transforms of $f_0$ and $f(t)$.  Indeed, by using the Fourier transform, we have
$$\psi_0=\frac{1}{2}e^{-\frac{|\xi|^2}{2}}+\frac{1}{12}\sum^3_{k=1}\left(e^{i\mathbf{e}_k\cdot\xi}+e^{-i\mathbf{e}_k\cdot\xi}\right)
=\frac{1}{2}e^{-\frac{|\xi|^2}{2}}+\frac{1}{6}\sum^3_{k=1}\cos\mathbf{e}_k\cdot\xi.$$
Set $\xi=(\xi_1,\xi_2,\xi_3)$, then
$$|1-\psi_0|=\frac{1}{2}(1-e^{-\frac{|\xi|^2}{2}})+\frac{1}{6}\sum^3_{k=1}(1-\cos\xi_k)=\frac{1}{2}(1-e^{-\frac{|\xi|^2}{2}})+\frac{1}{3}\sum^3_{k=1}\left(\sin\frac{\xi_k}{2}\right)^2
\leq\frac{1}{3}|\xi|^2.$$
This shows that
\begin{equation}\label{exist}
\psi_0\in \mathcal{K}^2,\,\text{and\,then}\,\psi(t)\in C([0, +\infty);\mathcal{K}^2).
\end{equation}
Now we want to prove the key coercive estimate \eqref{regular-kernel}.
Indeed, for $|\xi|\leq1$, then
$|\xi_k|\leq|\xi|\leq1$ for $k=1,2,3$, therefore, we have
$$\cos1\leq\cos\xi_k\leq1.$$
It follows that
\begin{align}\label{1}
1-|\psi_0|&\geq\frac{1}{2}(1-e^{-\frac{|\xi|^2}{2}})+\frac{1}{6}\sum^3_{k=1}(1-|\cos\xi_k|)\nonumber\\
&=\frac{1}{2}(1-e^{-\frac{|\xi|^2}{2}})+\frac{1}{3}\sum^3_{k=1}\left(\sin\frac{\xi_k}{2}\right)^2\nonumber\\
&\geq\frac{1}{3}\frac{4}{\pi^2}\sum^3_{k=1}\frac{\xi^2_k}{4}=\frac{1}{3\pi^2}|\xi|^2.
\end{align}
Besides, since $\psi_0,\psi(t)\in \mathcal{K}^2$ in \eqref{exist}, we can deduce from \eqref{re} and \eqref{stable} that,
\begin{align}\label{2}
|\psi(t)-\psi_0|
&\leq\int^t_0
\left|\int_{\mathbb{S}^2}b(\frac{\xi\cdot\sigma}{|\xi|})
\left(\varphi(\tau,\xi^+)\varphi(\tau,\xi^-)-\varphi(\tau,\xi)\right)d\sigma\right|d\tau\nonumber\\
&\lesssim\left(\int^t_0\|1-\varphi(\tau)\|_2d\tau\right)|\xi|^2\lesssim t|\xi|^2.
\end{align}
Therefore, for $|\xi|\leq1$ and $T>0$, there exist a positive constant $C$ dependent on $T$ such that, for $0<t<T$,
$$1-|\psi(t)|\geq 1-|\psi_0|-|\psi(t)-\psi_0|\geq \frac{1}{3\pi^2}|\xi|^2-Ct|\xi|^2,$$
Choosing the constant $T_1>0$ small enough, then for $0<t<T_1$ and $|\xi|\leq1$, we have
$$1-|\psi(t)|\gtrsim |\xi|^2.$$
On the other hand, for $|\xi|>1$,  by a proof similar to that in \eqref{1}, one can verify that
\begin{align}\label{3}
1-|\psi_0|&\geq\frac{1}{2}(1-e^{-\frac{|\xi|^2}{2}})+\frac{1}{6}\sum^3_{k=1}(1-|\cos\xi_k|)\nonumber\\
&\geq \frac{1}{2}(1-e^{-\frac{|\xi|^2}{2}})\geq \frac{1}{2}(1-e^{-\frac{1}{2}})
\end{align}
It follows from \eqref{2} that, for $|\xi|>1$, we have
$$\lim_{t\rightarrow0}|\psi(t)-\psi_0|=0.$$
Since $|\psi(t)|\leq |\psi_0|+|\psi(t)-\psi_0|$, then
$$1-\lim_{t\rightarrow0}|\psi(t)|\geq1-|\psi_0|-\lim_{t\rightarrow0}|\psi(t)-\psi_0|\geq \frac{1}{2}(1-e^{-\frac{1}{2}}).$$
We choose a $T_2$ small, such that $1-|\psi(t)|\geq C_{T_2}>0.$  In conclusion, set $T=\min(T_1, T_2)$, then for any $0<t<T$, the key estimate \eqref{regular-kernel} holds true.   By using the Proposition \ref{trick1} with the key estimate \eqref{regular-kernel}, we end the proof of the Example \ref{example}.

\end{proof}
\bigskip
\noindent {\bf Acknowledgements.}
The author would like to express his sincere thanks to
Prof. Chao-Jiang Xu for stimulating discussions and suggestions.
This research is supported by the fundamental research funds for South-Central University for Nationalities(No.CZQ16014).

\end{document}